\documentclass[11pt, twoside]{article}
\usepackage{amsfonts}
\usepackage{ulem}

\usepackage{amssymb}
\usepackage{amsmath}
\usepackage{amsthm}
\usepackage{xcolor}
\usepackage{mathrsfs}

\usepackage{verbatim}

\allowdisplaybreaks[4]

\allowdisplaybreaks

\def\R{{\mathbb R}}
\def\Z{{\mathbb Z}}
\def\N{{\mathbb N}}

\def\omiga{\omega}
\def\huaA{{\mathcal A}}

\def\Smn{S_{m,N}}
\def\lSmn{{\lambda}S_{m,N}}
\def\jsJ{\left|J\right|}
\def\jslJ{\left|\lambda{S_{m,N}}\right|}
\def\jssmn{\left|S_{m,N}\right|}
\def\kaf{\mathcal{X}}
\def\pfzy{\frac{1}{p}}

\def\M{\mathcal{M}}
\def\xieg{\setminus}
\def\kdjdz{\left|k\right|}

\def\R{{\mathbb R}}

\def\Z{{\mathbb Z}}

\def\z+{{\mathbb Z}_+}

\def\N{{\mathbb N}}

\def\bint{{\ifinner\rlap{\bf\kern.30em--}
\int\else\rlap{\bf\kern.35em--}\int\fi}\ignorespaces}

\def\sbint{{\ifinner\rlap{\bf\kern.32em--}
\hspace{0.078cm}\int\else\rlap{\bf\kern.45em--}\int\fi}\ignorespaces}

\newtheorem{theorem}{Theorem}[section]
\newtheorem{lemma}[theorem]{Lemma}

\newtheorem{proposition}[theorem]{Proposition}
\theoremstyle{definition}

\newtheorem{remark}[theorem]{Remark}

\newtheorem{definition}[theorem]{Definition}
\numberwithin{equation}{section}

\numberwithin{equation}{section}

\numberwithin{equation}{section}

\begin{document}

\arraycolsep=1pt

\title{\Large\bf The Hardy-Littlewood Maximal Operator on Discrete Weighted Morrey Spaces\footnotetext{\hspace{-0.35cm} {\it 2020
Mathematics Subject Classification}. {46B45, 42B35, 42B25.}
\endgraf{\it Key words and phrases.} weight, discrete Morrey space, discrete Hardy-Littlewood maximal operator, discrete Calder\'on-Zygmund decomposition.
\endgraf This project is supported by the National Natural Science Foundation of China (Grant No. 12261083).
\endgraf $^\ast$\,Corresponding author.
}}
\author{Xuebing Hao, Shuai Yang and Baode Li$^\ast$}
\date{ }
\maketitle

\vspace{-0.8cm}

\begin{center}
\begin{minipage}{13cm}\small
{\noindent{\bf Abstract.}
In this paper, we introduce a discrete version of weighted Morrey spaces, and discuss the inclusion relations of these spaces. In addition, we obtain the boundedness of discrete weighted Hardy-Littlewood maximal operators on discrete weighted Lebesgue spaces by establishing a discrete Calder\'on-Zygmund decomposition for weighted $l^1$-sequences. Furthermore, the necessary and sufficient conditions for the boundedness of the discrete Hardy-Littlewood maximal operators on discrete weighted Morrey spaces are discussed. Particularly, necessary and sufficient conditions are also discussed for the discrete power weights.}
\end{minipage}
\end{center}

\section{Introduction\label{s1}}
In 1938, the classical Morrey spaces were introduced by Morrey in \cite{dimorrey} to investigate the local behavior of solutions to second order elliptic partial differential equations. In 1987, Chiarenza and Frasca \cite{1987} showed the boundedness of the Hardy-Littlewood maximal operator on the Morrey spaces. In 2009, Komori and Shirai \cite{jdM} defined a weighted Morrey space and investigated the boundedness of the Hardy-Littlewood maximal operator on this space.

Let $m=(m_1,\,\cdots,\,m_d)\in \Z^d$, $N \in\N$ and $S_{m,N}:=\{k\in\Z^d: \|k-m\|_\infty\le N\}$, where as usual $\|(m_1,\,\cdots,\,m_d)\|_\infty:=\max\{|m_i|: 1\le i\le d\}$ for every $m\in \Z^d$. Then $\left| S_{m,N} \right|=(2N+1)^d$$-$the cardinality of $S_{m,N}$.

\begin{definition}
Let $1\le p\le q<\infty$. The {\it discrete Morrey space} is defined by $l^p_q=l^p_q(\Z)$ the set of sequences $x=\{x(k)\}_{k\in\Z}$ taking values in $\R$ such that
\begin{equation*}
\|x\|_{l^p_q}:=\sup\limits_{m\in\Z,N\in\N}|S_{m,N}|^{\frac{1}{q}-\frac{1}{p}}\left(\sum\limits_{k\in S_{m,N}}|x(k)|^p\right)^\pfzy<\infty.
\end{equation*}
\end{definition}

In 2018, Gunawan, Kikianty and Schwanke \cite{lsmyxz18} studied the discrete Morrey spaces and their generalizations. In addition, discrete Morrey spaces can be used to study the properties of discrete operators and the properties of the spaces themselves. For more results on discrete Morrey spaces, we lead the reader to \cite{2022Abe,2022A,2021A,2020H,2019K,2023W}.

\begin{definition}
Let $x=\{x(k)\}_{k\in\Z}\subset\R$ be a sequence. The {\it discrete Hardy-Littlewood maximal operator} $\M$ is defined by
\begin{equation*}
\M x(m):=\sup_{N\in \N}\frac{1}{\jssmn}\sum_{k\in \Smn}\left|x(k)\right|,\quad m\in \Z.
\end{equation*}
\end{definition}

In 2019, Gunawan and Schwanke \cite{jdszzlsm} discussed the boundedness of the discrete Hardy-Littlewood maximal operators on discrete Morrey spaces of arbitrary dimension.

\begin{theorem}\label{1.1}
Let $1<p\le q<\infty$. For all $x\in l^p_q(\Z^d)$ we have $\M x\in l^p_q(\Z^d)$ and there exists a constant $C>0$ such that $\|\M x\|_{l^p_q(\Z^d)}\le C\|x\|_{l^p_q(\Z^d)}$ holds for all $x\in l^p_q(\Z^d)$.
\end{theorem}

On the other hand, it is very important to study weighted estimates for all kinds of discrete operators in harmonic analysis. On the discrete weighted $l^p$ spaces, in 2021, when $\Z$ is restricted to $\Z_+$, the boundedness of the discrete Hardy-Littlewood maximal operators was obtained via Hardy operator by Saker and Agarwal \cite{lsq}. For more studies of discrete weighted Lebesgue spaces, we refer the reader to \cite{wiener20,fu,2019sa,2021sa}. However, discrete weighted Morrey spaces have not been studied.

Thus, we generate the following natural question:
\begin{itemize}
\item[$Q).$] Is it possible to prove Theorem \ref{1.1} for weighted sequences in the discrete weighted Morrey spaces when $d=1$ ?
\end{itemize}

Our aim is to give an affirmative answer to the question which promotes the development of the applications of discrete Muckenhoupt weights and discrete Morrey spaces on discrete harmonic analysis.

This article is organized as follows. In Sect.\,\ref{s2}, we introduce the discrete weighted Morrey space, discrete $\huaA_p$ weight and discrete Hardy-Littlewood maximal operator. Then we prove some inclusion relations of discrete weighted Morrey spaces and discrete $\huaA_p$ weights. In Sect.\,\ref{s3}, inspired by the continuous Calder\'on-Zygmund decomposition for weighted $L^1(\R^n)$ of Garcia-Cuerva and Rubio de Francia \cite{norm}, we further give the discrete Calder\'on-Zygmund decomposition for weighted $l^1(\Z)$ (see Theorem \ref{t3.2} below). As applications, in Sect.\,\ref{s4.1}, we obtain strong type and weak type inequalities for the discrete weighted Hardy-Littlewood maximal operators on discrete weighted Lebesgue spaces, and the proofs are quite different from the one restricted to $\Z_+$ (\cite[Lemma 3]{lsq}). Precisely, Yang inequality and some properties of Hardy operator are crucial tools for the bounded results on $\Z_+$ while the discrete Calder\'on-Zygmund decomposition plays an important role for the bounded results on $\Z$. In Sect.\,\ref{s4.2}, applying these results in Sect.\,\ref{s4.1}, we further obtain the necessary and sufficient conditions for the boundedness of the discrete Hardy-Littlewood maximal operators on discrete weighted Morrey spaces. In Subsect.\,\ref{s4.2.2}, lighted by \cite{2020D}, the necessary and sufficient conditions are also obtained for the power weights.

\section*{Notation}

\ \,\ \ $-\Z:$ set of integers;

$-\Z_+:=\{1,2,3,\cdots\}$;

$-\N:=\{0,1,2,\cdots\}$;


$-\kaf_{J}:$ the characteristic function of interval $J$;

$-r':$ the conjugate exponent of $r$, namely, $\frac{1}{r}+\frac{1}{r'}=1$;

$-C:$ a positive constant which is independent of the main parameters, but it may vary from line to line.




\section{Preliminaries \label{s2}}
In this section, we first introduce the discrete weighted Morrey space and discrete Muckenhoupt class. Then we discuss some properties of these spaces and Muckenhoupt classes. Let us begin with the definition of discrete weighted Morrey space.

A discrete weight on $\Z$ is a sequence $\omiga = \{\omiga(k)\}_{k\in\Z}$ of positive real numbers.

\begin{definition}\label{d2.1}
Given $1\le p\le q<\infty$. For two discrete weights $\omiga$ and $v$, we define the {\it discrete weighted Morrey space} $l{^p_q}_{(\omiga,v)}=l{^p_q}_{(\omiga,v)}(\Z)$ to the space of all sequences $x=\{x(k)\}_{k\in\Z}\subset\R$ for which
\begin{equation*}
\|x\|_{l{^p_q}_{(\omiga,v)}}:=\sup_{m\in \Z,N\in\N}(v(S_{m,N}))^{\frac{1}{q}-\frac{1}{p}}\bigg(\sum_{k\in S_{m,N}}\left|x(k)\right|^p\omiga(k)\bigg)^\frac{1}{p}<\infty,
\end{equation*}
where $v(S_{m,N}):=\sum\limits_{k\in S_{m,N}}v(k)$. Particularly, when $p=q$,
\begin{equation*}
\|x\|_{l{^p_q}_{(\omiga,v)}}=\|x\|_{l^p_\omiga}:=\left(\sum\limits_{k\in\Z}|x(k)|^p\omiga(k)\right)^\frac{1}{p}.
\end{equation*}
\end{definition}

\begin{remark}\label{2.2}
When $\omiga= v\equiv 1$, ${l^p_q}_{(\omiga,v)}(\Z)$ coincides with discrete Morrey space $l^p_q(\Z)$. In addition, since $\|x\|_{l{^p_q}}:=\sup\limits_{m\in \Z,N\in\N}\left| S_{m,N}\right|^{\frac{1}{q}-\frac{1}{p}}\bigg(\sum\limits_{k\in S_{m,N}}\left|x(k)\right|^p\bigg)^\frac{1}{p}$ is a norm on $l{^p_q}$ (see \cite[Proposition 2.2]{lsmyxz18}). Similarly, we can obtain $\|x\|_{l{^p_q}_{(\omiga,v)}}$ is also a norm on $l{^p_q}_{(\omiga,v)}$. Furthermore, $l{^p_q}_{(\omiga,v)}$ is a Banach space with respect to the norm $\|x\|_{l{^p_q}_{(\omiga,v)}}$.
\end{remark}

\begin{remark}\label{2.3}
\begin{itemize}
\item[\rm(i)] When discrete weight $v=\omiga$, $l{^p_q}_{(\omiga,v)}$ reduces to $l{^p_q}_{(\omiga)}$:
\begin{equation*}
\|x\|_{l{^p_q}_{(\omiga)}}:=\sup_{m\in \Z,N\in\N}(\omiga(S_{m,N}))^{\frac{1}{q}-\frac{1}{p}}\bigg(\sum_{k\in S_{m,N}}\left|x(k)\right|^p\omiga(k)\bigg)^\frac{1}{p}.
\end{equation*}
\item[\rm(ii)] When $1\le p<q=\infty$, $l{^p_q}_{(\omiga,v)}$ reduces to $l{^p_\infty}_{(\omiga,v)}$:
\begin{equation*}
\|x\|_{l{^p_\infty}_{(\omiga,v)}}:=\sup_{m\in \Z,N\in\N}(v(S_{m,N}))^{-\frac{1}{p}}\bigg(\sum_{k\in S_{m,N}}\left|x(k)\right|^p\omiga(k)\bigg)^\frac{1}{p}.
\end{equation*}
\end{itemize}
\end{remark}

Proposition \ref{m2.4} will be useful in studying the relations between two discrete weighted Morrey spaces.

\begin{proposition}\label{m2.4}
Let $\omiga$ and $v$ be discrete weights.
\begin{itemize}
\item[\rm(i)] If $1\le p\le q<\infty$ and $\omiga(k)\le Cv(k)$ for every $k\in \Z$, then $l{^p_q}_{(v,\omiga)}\subseteq l{^p_q}_{(\omiga,v)}$ with $\|x\|_{l{^p_q}_{(\omiga,v)}}\le C\|x\|_{l{^p_q}_{(v,\omiga)}}$ for every $x\in l{^p_q}_{(v,\omiga)}$.
\item[\rm(ii)] If $1\le p<q=\infty$ and $v=\omiga$, then $l^\infty\subseteq {l^p_\infty}_{(\omiga)}$ with $\|x\|_{{l^p_\infty}_{(\omiga)}}\le\|x\|_{l^\infty}$ for any $x\in l^\infty$.
\item[\rm(iii)] If $1<p<q<\infty$ and $v=\omiga$, then ${l^q_\omiga}\subseteq {l^p_q}_{(\omiga)}$ and $\|x\|_{{l^p_q}_{(\omiga)}}\le \|x\|_{{l^q_\omiga}}$ for any $x\in {l^q_\omiga}$.
\end{itemize}
\end{proposition}

\begin{proof}
$\rm(i)$ Since $\omiga(k)\le Cv(k)$ for every $k\in \Z$, we have $\left|x(k)\right|^p\omiga(k)\le C\left|x(k)\right|^pv(k)$ for every $k\in \Z$. Therefore, for any $m\in \Z$ and $N\in \N$, we obtain
\begin{equation}\label{eq2.1}
\bigg(\sum_{k\in S_{m,N}}\left|x(k)\right|^p\omiga(k)\bigg)^\frac{1}{p}\le C\bigg(\sum_{k\in S_{m,N}}\left|x(k)\right|^pv(k)\bigg)^\frac{1}{p}.
\end{equation}
In addition, by $\frac{1}{q}-\frac{1}{p}\le 0$, we have
\begin{equation}\label{eq2.2}
\bigg(\sum_{k\in S_{m,N}}v(k)\bigg)^{\frac{1}{q}-\frac{1}{p}}\le C\bigg(\sum_{k\in S_{m,N}}\omiga(k)\bigg)^{\frac{1}{q}-\frac{1}{p}}.
\end{equation}
According to (\ref{eq2.1}) and (\ref{eq2.2}), we obtain
\begin{align*}
&\bigg(\sum_{k\in S_{m,N}}v(k)\bigg)^{\frac{1}{q}-\frac{1}{p}}\bigg(\sum_{k\in S_{m,N}}\left|x(k)\right|^p\omiga(k)\bigg)^\frac{1}{p}\\
\le& C\bigg(\sum_{k\in S_{m,N}}\omiga(k)\bigg)^{\frac{1}{q}-\frac{1}{p}}\bigg(\sum_{k\in S_{m,N}}\left|x(k)\right|^pv(k)\bigg)^\frac{1}{p}.
\end{align*}
This proves $\rm(i)$.


$\rm(ii)$ Let $1\le p<q=\infty$. If $v=\omiga$ and $x\in l^\infty$, then for any $m\in\Z$ and $N\in\N$, we have
\begin{equation*}
\omiga(S_{m,N})^{-\frac{1}{p}}\left(\sum_{k\in S_{m,N}}\left|x(k)\right|^p\omiga(k)\right)^\pfzy\le \omiga(S_{m,N})^{-\frac{1}{p}}\|x\|_{l^\infty}\omiga(S_{m,N})^\pfzy=\|x\|_{l^\infty},
\end{equation*}
which implies that $\|x\|_{{l^p_\infty}_{(\omiga)}}\le\|x\|_{l^\infty}$ and hence $l^\infty\subseteq {l^p_\infty}_{(\omiga)}$.

Since the proof of $\rm(iii)$ is derived immediately from H\"older's inequality, we omit it.
We finish the proof of Proposition \ref{m2.4}.
\end{proof}

\begin{definition}
For $1\le p\le q<\infty$, we defined the {\it weighted weak type discrete Morrey space} ${wl^p_q}_{(\omiga)}$ to be the set of sequences $x=\{x(k)\}_{k\in\Z}$ taking values in $\R$ such that $\|x\|_{{wl^p_q}_{(\omiga)}}<\infty$, where $\|\cdot\|_{{wl^p_q}_{(\omiga)}}$ is given by
\begin{equation*}
\|x\|_{{wl^p_q}_{(\omiga)}}:=\sup_{m\in\Z,N\in\N,\lambda>0}\omiga(S_{m,N})^{\frac{1}{q}-\frac{1}{p}}\lambda\omiga(\{k\in S_{m,N}: |x(k)|>\lambda\})^\frac{1}{p}.
\end{equation*}
Note that when $\omiga=1$, ${wl^p_q}_{(\omiga)}=:wl^p_q$, which is a weak type discrete Morrey space (see \cite[Page 4]{lsmyxz18}).
\end{definition}

\begin{proposition}\label{p2.2}
Let $1\le p\le q<\infty$ and $\omiga$ be a discrete weight on $\Z$. Then ${l^p_q}_{(\omiga)}\subseteq {wl^p_q}_{(\omiga)}$ with $\|x\|_{{wl^p_q}_{(\omiga)}}\le\|x\|_{{l^p_q}_{(\omiga)}}$ for every $x\in {l^p_q}_{(\omiga)}$.
\end{proposition}

\begin{proof}
For $x\in {l^p_q}_{(\omiga)}$, $m\in\Z$, $N\in\N$ and $\lambda>0$, we have
\begin{small}
\begin{equation*}
\begin{aligned}
&\omiga(S_{m,N})^{\frac{1}{q}-\frac{1}{p}}\lambda\omiga(\{k\in S_{m,N}: |x(k)|>\lambda\})^\frac{1}{p}\\
\le &\omiga(S_{m,N})^{\frac{1}{q}-\frac{1}{p}}\left(\sum\limits_{k\in S_{m,N}, |x(k)|>\lambda}|x(k)|^p\omiga(k)\right)^\frac{1}{p}
\le \omiga(S_{m,N})^{\frac{1}{q}-\frac{1}{p}}\left(\sum\limits_{k\in S_{m,N}}|x(k)|^p\omiga(k)\right)^\frac{1}{p}.
\end{aligned}
\end{equation*}
\end{small}\\
Taking the supremum over $m\in\Z, N\in\N$ and $\lambda>0$, we finish the proof of Proposition \ref{p2.2}.
\end{proof}

In the following, referring to the definition of discrete Muckenhoupt class restricted to $\Z_+$ (see \cite[Page 2]{lsq2}), we give the definition of discrete Muckenhoupt class on $\Z$.

By interval $J$, we mean a finite subset of $\Z$ consisting of consecutive integers, i.e., $J=\{a,a+1,\dots,a+n\}$, $a, n\in \Z$, and $\left|J\right|$ stands for its cardinality.

\begin{definition}\label{d2.6}
A discrete weight $\omiga$ is said to belong to the {\it discrete Muckenhoupt class} $\huaA_1=\huaA_1(\Z)$ if
\begin{equation*}
\|\omiga\|_{\huaA_1(\Z)}:=\sup\limits_{J\subset \Z}\frac{1}{\jsJ}\left(\frac{1}{\inf\limits_{k\in J}\omiga(k)}\sum\limits_{k\in J}\omiga(k)\right)<\infty.
\end{equation*}
For $1<p<\infty$, a discrete weight $\omiga$ is said to belong to the {\it discrete Muckenhoupt class} $\huaA_p=\huaA_p(\Z)$ if
\begin{equation*}
\|\omiga\|_{\huaA_p(\Z)}:=\sup\limits_{J\subset\Z}\left(\frac{1}{\left|J\right|}\sum_J\omiga\right)\left(\frac{1}
{\left|J\right|}\sum_J\omiga^{\frac{-1}{p-1}}\right)
^{p-1}<\infty,
\end{equation*}
where $\|\omiga\|_{\huaA_p(\Z)}$ denotes the norm of weight $\omiga$ and $J$ is any bounded interval in $\Z$. Define $\huaA_\infty:=\mathop{\cup}\limits_{1\le p<\infty}\huaA_p$.
\end{definition}

Some basic properties of discrete $\huaA_p$ weights are given as follows: Their proofs are the same as corresponding results on the real line \cite{shu}. For the sake of the completeness, some of the proofs are included.

\begin{proposition}\label{p1.1}
Let $\omiga$ be a discrete weight on $\Z$. Then the following statements are equivalent:
\begin{itemize}
\item[\rm(i)] $\omiga\in\huaA_1$;
\item[\rm(ii)] $\left(\frac{1}{\jsJ}\sum\limits_{k\in J}\omiga(k)\right)\|\omiga^{-1}\|_{l^\infty(J)}\le C$, for every bounded interval $J\subset \Z$. Here, $\|\omiga\|_{l^\infty(J)}:=\max\limits_{k\in J}\left|\omiga(k)\right|$.
\end{itemize}
\end{proposition}

\begin{proposition}\label{m2.8}
Let $\omiga$ be a discrete weight on $\Z$. Then the following statements are equivalent:
\begin{itemize}
\item[\rm(i)] $\omiga\in\huaA_p$ $(1<p<\infty)$;
\item[\rm(ii)] $\frac{1}{\left|J\right|}\sum\limits_{k\in J}\left|x(k)\right|\le C\left(\frac{1}{\omiga(J)}\sum\limits_{k\in J}\left|x(k)\right|^p\omiga(k)\right)^{\frac{1}{p}}$, $x\in l{^p_{\omiga}}(\Z)$.
\end{itemize}
\end{proposition}

\begin{proof}
First, we prove that $\rm(i)\Rightarrow\rm(ii)$. By H\"older's inequality and the definition of $\huaA_p$ $(1<p<\infty)$, we have
\begin{align*}
\frac{1}{\left|J\right|}\sum\limits_{k\in J}\left|x(k)\right|&= \frac{1}{\left|J\right|}\sum\limits_{k\in J}\left|x(k)\right|{\omiga(k)}^{\frac{1}{p}}{\omiga(k)}^{-\frac{1}{p}}\\
&\le \frac{1}{\left|J\right|}\left(\sum\limits_{k\in J}{\left|x(k)\right|}^p{\omiga(k)}\right)^{\frac{1}{p}}\left(\sum\limits_{k\in J}\omiga(k)^{-\frac{1}{p-1}}\right)^{\frac{p-1}{p}}\\
&\le C\left(\sum\limits_{k\in J}{\left|x(k)\right|}^p{\omiga(k)}\right)^{\frac{1}{p}}\left(\sum\limits_{k\in J}\omiga(k)\right)^{-\frac{1}{p}}\\
&= C\left(\frac{1}{\omiga(J)}\sum\limits_{k\in J}{\left|x(k)\right|}^p{\omiga(k)}\right)^{\frac{1}{p}}.
\end{align*}

Next, let us show $\rm(ii)\Rightarrow\rm(i)$. Let $x(k)={\omiga(k)}^{-\frac{1}{p-1}}$. According to $\rm(ii)$, we obtain
\begin{equation*}\begin{array}{rcl}
\frac{1}{\left|J\right|}\sum\limits_{k\in J}{\omiga(k)}^{-\frac{1}{p-1}}&\le& C\left(\frac{1}{\omiga(J)}\sum\limits_{k\in J}{\omiga(k)}^{-\frac{p}{p-1}}\omiga(k)\right)^\pfzy
\end{array}\end{equation*}
\begin{equation*}
\Longrightarrow\quad\frac{1}{\left|J\right|}\bigg(\sum\limits_{k\in J}\omiga(k)\bigg)^{\frac{1}{p}}\bigg(\sum\limits_{k\in J}{\omiga(k)}^{-\frac{1}{p-1}}\bigg)^\frac{p-1}{p}\le C
\end{equation*}
\begin{align*}
& \Longrightarrow\quad\bigg(\frac{1}{\left|J\right|}\sum_{k\in J}\omiga(k)\bigg)\bigg(\frac{1}{\left|J\right|}\sum_{k\in J}\omiga(k)^{\frac{-1}{p-1}}\bigg)^{p-1}\le C\\
& \Longrightarrow\quad \omiga\in\huaA_p.
\end{align*}
We finish the proof of Proposition \ref{m2.8}.
\end{proof}

From Propositions \ref{p1.1} and \ref{m2.8}, we deduce the following proposition.

\begin{proposition}\label{m2.9}
Let $\omiga\in\huaA_p$ $(1\le p<\infty)$, $\lambda\in\Z_+$, $m\in\Z$, $N\in\N$ and $\lambda S_{m,N}:=\{k\in\Z:\left|k-m\right|\le \lambda N\}$. For each $N\in\N$, consider the collection of disjoint intervals of cardinality $2^N$,
\begin{equation*}
I_N:=\{I_{N,j}\}:=\{\{(j-1)2^N+1,\cdots,j2^N\}\}_{j\in\Z}.
\end{equation*}
For any interval $I\in I_N$ and $n\ge 2$ we define
\begin{equation*}
\begin{aligned}
&nLI:=\{(j-n)2^N+1,\cdots,j2^N\},\quad \left|nLI\right|=n2^N;\\
&nRI:=\{(j-1)2^N+1,\cdots,(j+n-1)2^N\},\quad \left|nRI\right|=n2^N;\\
&(2n-1)I:=\{(j-n)2^N+1,\cdots,(j+n-1)2^N\},\quad \left|(2n-1)I\right|=(2n-1)2^N.
\end{aligned}
\end{equation*}
Then the following holds:
\begin{itemize}
\item[\rm(i)]
There exists a constant $C>0$ such that
\begin{equation*}
\sum_{k\in\lambda S_{m,N}}\omiga(k)\le \left(\frac{3}{2}C\right)^p\lambda^p\sum_{k\in S_{m,N}}\omiga(k).
\end{equation*}
\item[\rm(ii)]
There exists a constant $C>0$ such that
\begin{equation*}
\begin{aligned}
\sum_{k\in nLI}\omiga(k)&\le Cn^p\sum_{k\in I}\omiga(k),\\
\sum_{k\in nRI}\omiga(k)&\le Cn^p\sum_{k\in I}\omiga(k),\\
\sum_{k\in (2n-1)I}\omiga(k)&\le C(2n-1)^p\sum_{k\in I}\omiga(k).
\end{aligned}
\end{equation*}
\end{itemize}
\end{proposition}

\begin{proof}
$\rm(i)$ When $p=1$, according to Definition \ref{d2.6}, we obtain
\begin{equation*}
\frac{1}{\jslJ}\sum_{k\in\lSmn}\omiga(k)\le C\inf_{k\in\lSmn}\omiga(k)\le C\inf_{k\in\Smn}\omiga(k)
\end{equation*}
\begin{align*}
\Longrightarrow\quad\sum\limits_{k\in\lSmn}\omiga(k)&\le C\jslJ\inf\limits_{k\in\Smn}\omiga(k)\\
&=C\frac{\jslJ}{\jssmn}\jssmn\inf\limits_{k\in\Smn}\omiga(k)\\
&\le C\frac{\jslJ}{\jssmn}\sum\limits_{k\in\Smn}\omiga(k)\\
&= C\frac{2\lambda N+1}{2N+1}\sum\limits_{k\in\Smn}\omiga(k)\\
&\le \frac{3}{2}C\lambda\sum\limits_{k\in\Smn}\omiga(k).
\end{align*}

When $p>1$, let $x(k):=\mathcal{X}_{S_{m,N}}(k)$. By Proposition \ref{m2.8}, one has
\begin{small}
\begin{equation*}
\begin{aligned}
&\frac{1}{\jslJ}\sum\limits_{k\in \lSmn}\kaf_{\Smn}(k)\le C\left(\frac{1}{\omiga(\lambda S_{m,N})}\sum\limits_{k\in\lSmn}\left|\kaf_{\Smn}(k)\right|^p\omiga(k)\right)^{\frac{1}{p}}\\
\Longrightarrow &\frac{\jssmn}{\jslJ}\left(\sum\limits_{k\in\lSmn}\omiga(k)\right)^{\pfzy}\le
C\left(\sum\limits_{k\in\Smn}\omiga(k)\right)^{\pfzy}\\
\Longrightarrow &\left(\sum\limits_{k\in\lSmn}\omiga(k)\right)^{\pfzy}\le C\frac{{2{\lambda}N+1}}{2N+1}\left(\sum\limits_{k\in\Smn}\omiga(k)\right)^{\pfzy}\\
\Longrightarrow &\left(\sum\limits_{k\in\lSmn}\omiga(k)\right)^{\pfzy}\le C\frac{3}{2}\lambda\left(\sum\limits_{k\in\Smn}\omiga(k)\right)^{\pfzy}\\
\Longrightarrow &\sum_{k\in\lSmn}\omiga(k)\le\left(\frac{3}{2}C\right)^p\lambda^p\sum_{k\in\Smn}\omiga(k).
\end{aligned}
\end{equation*}
\end{small}

The proof of $\rm(ii)$ is similar to that of $\rm(i)$, so the details are omitted. This finishes the proof of Proposition \ref{m2.9}.
\end{proof}

\begin{remark}\label{r2.10}
If $\omiga\equiv C$, then $C\in \huaA_p$ $(p\ge 1)$. Particularly, if $\omiga\equiv 1$, then by Proposition \ref{m2.9}, we have $\jslJ\le C\lambda\jssmn$, i.e., the counting measure is a doubling measure.
\end{remark}

Similar to the definitions of discrete $\huaA_p(\Z)$ ($1\le p< \infty$) weights, replacing $\Z$ by $\N$, we can also give the definitions of discrete $\huaA_p(\N)$ weights. The following Proposition \ref{m2.15} discussed the relationship of $\huaA_p(\Z)$ and $\huaA_p(\N)$.

\begin{proposition}\label{m2.15}
If $\omiga\in\huaA_p(\N)$ $(1\le p<\infty)$, then $\widetilde{\omiga}(\cdot):=\omiga(\left|\cdot\right|)\in\huaA_p(\Z)$ $(1\le p<\infty)$ with
$$
\|\widetilde{\omiga}\|_{\huaA_p(\Z)}\le
\begin{cases}
4\|\omiga\|_{\huaA_1(\N)}, \quad p = 1,\\
4C_{2,p}\|\omiga\|_{\huaA_p(\N)}, \quad p>1.
\end{cases}
$$
\end{proposition}

\begin{proof}
For every bounded interval $J\subset\Z$, if $J\subset\N$ or $J\subset\Z_-\cup\{0\}$, then the result is obvious. Now we only need to discuss the case where the interval $J$ contains both negative integers, zero and positive integers.

Let $J=J_1\cup J_2$ with $J_1\subset\Z_-$ and $J_2\subset\N$. When $p=1$, we have
\begin{equation}\label{eq2.11}
\begin{aligned}
\max\limits_{k\in J}\frac{1}{\omiga(\kdjdz)}&=\max\limits_{k\in -J_1\cup J_2}\frac{1}{\omiga(k)}\\
&\le\max\left\{\left(\max\limits_{k\in -J_1}\frac{1}{\omiga(k)}\right),\left(\max\limits_{k\in J_2}\frac{1}{\omiga(k)}\right)\right\}\\
&\le\max\limits_{k\in -J_1}\frac{1}{\omiga(k)}+\max\limits_{k\in J_2}\frac{1}{\omiga(k)}.
\end{aligned}
\end{equation}
From Proposition \ref{p1.1} and (\ref{eq2.11}), it follows that
\begin{small}
\begin{equation}\label{eq2.6}
\begin{aligned}
\frac{1}{\jsJ}\sum\limits_{k\in J}\omiga(\kdjdz)\|\omiga^{-1}\|_{l^\infty}
=&\frac{1}{\jsJ}\left(\sum\limits_{k\in -J_1}\omiga(k)+\sum\limits_{k\in J_2}\omiga(k)\right)\max\limits_{k\in J}\frac{1}{\omiga(\kdjdz)}\\
\le &\left(\frac{1}{\jsJ}\sum\limits_{k\in -J_1}\omiga(k)+\frac{1}{\jsJ}\sum\limits_{k\in J_2}\omiga(k)\right)\left(\max\limits_{k\in -J_1}\frac{1}{\omiga(k)}+\max\limits_{k\in J_2}\frac{1}{\omiga(k)}\right)\\
=& \left(\frac{1}{\jsJ}\sum\limits_{k\in -J_1}\omiga(k)\right)\max\limits_{k\in -J_1}\frac{1}{\omiga(k)}+\left(\frac{1}{\jsJ}\sum\limits_{k\in -J_1}\omiga(k)\right)\max\limits_{k\in J_2}\frac{1}{\omiga(k)}\\
& +\left(\frac{1}{\jsJ}\sum\limits_{k\in J_2}\omiga(k)\right)\max\limits_{k\in -J_1}\frac{1}{\omiga(k)}+\left(\frac{1}{\jsJ}\sum\limits_{k\in J_2}\omiga(k)\right)\max\limits_{k\in J_2}\frac{1}{\omiga(k)}\\
\overset{(a)}{\le} &\left(\frac{1}{\left|-J_1\right|}\sum\limits_{k\in -J_1}\omiga(k)\right)\|\omiga^{-1}\|_{l^\infty(-J_1)}+\left(\frac{1}{\left|J_2\right|}\sum\limits_{k\in J_2}\omiga(k)\right)\|\omiga^{-1}\|_{l^\infty(J_2)}\\
& +\left(\frac{1}{\jsJ}\sum\limits_{k\in -J_1}\omiga(k)\right)\max\limits_{k\in J_2}\frac{1}{\omiga(k)}+\left(\frac{1}{\jsJ}\sum\limits_{k\in J_2}\omiga(k)\right)\max\limits_{k\in -J_1}\frac{1}{\omiga(k)}\\
\overset{(b)}{\le} &2\|\omiga\|_{\huaA_1(\N)}+\left(\frac{1}{\jsJ}\sum\limits_{k\in -J_1}\omiga(k)\right)\max\limits_{k\in J_2}\frac{1}{\omiga(k)}+\left(\frac{1}{\jsJ}\sum\limits_{k\in J_2}\omiga(k)\right)\max\limits_{k\in -J_1}\frac{1}{\omiga(k)},
\end{aligned}
\end{equation}
\end{small}
where $(a)$ is due to $\left|-J_1\right|\le\jsJ$ and $\left|J_2\right|\le\jsJ$, and $(b)$ is from $\omiga\in\huaA_1(\N)$.

Let $\Gamma_1:=\left(\frac{1}{\jsJ}\sum\limits_{k\in -J_1}\omiga(k)\right)\max\limits_{k\in J_2}\frac{1}{\omiga(k)}+\left(\frac{1}{\jsJ}\sum\limits_{k\in J_2}\omiga(k)\right)\max\limits_{k\in -J_1}\frac{1}{\omiga(k)}$. To estimate $\Gamma_1$, we consider the following two cases.

\textbf{Case I:} If $-J_1\subset J_2$,
then
\begin{small}
\begin{equation*}
\begin{aligned}
\Gamma_1&\le\left(\frac{1}{\left|J_2\right|}\sum\limits_{k\in J_2}\omiga(k)\right)\max\limits_{k\in J_2}\frac{1}{\omiga(k)}+\left(\frac{1}{\left|J_2\right|}\sum\limits_{k\in J_2}\omiga(k)\right)\max\limits_{k\in J_2}\frac{1}{\omiga(k)}\\
&\le 2\|\omiga\|_{\huaA_1(\N)}.
\end{aligned}
\end{equation*}
\end{small}

\textbf{Case II:} If $J_2\subset -J_1\cup\{0\}=:J_3$, then
\begin{small}
\begin{equation*}
\begin{aligned}
\Gamma_1&\le\left(\frac{1}{\left|J_3\right|}\sum\limits_{k\in J_3}\omiga(k)\right)\max\limits_{k\in J_3}\frac{1}{\omiga(k)}+\left(\frac{1}{\left|J_2\right|}\sum\limits_{k\in J_3}\omiga(k)\right)\max\limits_{k\in J_3}\frac{1}{\omiga(k)}\\
&\le 2\|\omiga\|_{\huaA_1(\N)}.
\end{aligned}
\end{equation*}
\end{small}
Combining all the estimates for Case I and Case II of $\Gamma_1$, (\ref{eq2.6}) implies that
\begin{equation*}
\frac{1}{\jsJ}\sum\limits_{k\in J}\omiga(\kdjdz)\|\omiga^{-1}\|_{l^\infty}\le 4\|\omiga\|_{\huaA_1(\N)}.
\end{equation*}

When $p>1$, we have
\begin{small}
\begin{align}\label{eq2.4}
&\bigg(\frac{1}{\jsJ}\sum\limits_{k\in J}\omiga(\kdjdz)\bigg)\bigg(\frac{1}{\jsJ}\sum\limits_{k\in J}\omiga(\kdjdz)^{-\frac{1}{p-1}}\bigg)^{p-1}\nonumber\\
=&\bigg[\frac{1}{\jsJ}\bigg(\sum\limits_{k\in J_1}\omiga(\kdjdz)+\sum\limits_{k\in J_2}\omiga(k)\bigg)\bigg]\bigg[\frac{1}{\jsJ}\bigg(\sum\limits_{k\in J_1}\omiga(\kdjdz)^{-\frac{1}{p-1}}+\sum\limits_{k\in J_2}\omiga(k)^{-\frac{1}{p-1}}\bigg)\bigg]^{p-1}\nonumber\\
=&\bigg(\frac{1}{\jsJ}\sum\limits_{k\in -J_1}\omiga(k)+\frac{1}{\jsJ}\sum\limits_{k\in J_2}\omiga(k)\bigg)\bigg(\frac{1}{\jsJ}\sum\limits_{k\in -J_1}\omiga(k)^{-\frac{1}{p-1}}+\frac{1}{\jsJ}\sum\limits_{k\in J_2}\omiga(k)^{-\frac{1}{p-1}}\bigg)^{p-1}\nonumber\\
\overset{(c)}{\le}&C_{2,p}\bigg(\frac{1}{\jsJ}\sum\limits_{k\in -J_1}\omiga(k)+\frac{1}{\jsJ}\sum\limits_{k\in J_2}\omiga(k)\bigg)\bigg[\bigg(\frac{1}{\jsJ}\sum\limits_{k\in -J_1}\omiga(k)^{-\frac{1}{p-1}}\bigg)^{p-1}+\bigg(\frac{1}{\jsJ}\sum\limits_{k\in J_2}\omiga(k)^{-\frac{1}{p-1}}\bigg)^{p-1}\bigg]\nonumber\\
=&C_{2,p}\bigg[\bigg(\frac{1}{\jsJ}\sum\limits_{k\in -J_1}\omiga(k)\bigg)\bigg(\frac{1}{\jsJ}\sum\limits_{k\in -J_1}\omiga(k)^{-\frac{1}{p-1}}\bigg)^{p-1}+\bigg(\frac{1}{\jsJ}\sum\limits_{k\in -J_1}\omiga(k)\bigg)\bigg(\frac{1}{\jsJ}\sum\limits_{k\in J_2}\omiga(k)^{-\frac{1}{p-1}}\bigg)^{p-1}\nonumber\\
&+\bigg(\frac{1}{\jsJ}\sum\limits_{k\in J_2}\omiga(k)\bigg)\bigg(\frac{1}{\jsJ}\sum\limits_{k\in -J_1}\omiga(k)^{-\frac{1}{p-1}}\bigg)^{p-1}+\bigg(\frac{1}{\jsJ}\sum\limits_{k\in J_2}\omiga(k)\bigg)\bigg(\frac{1}{\jsJ}\sum\limits_{k\in J_2}\omiga(k)^{-\frac{1}{p-1}}\bigg)^{p-1}\bigg]\nonumber\\
\overset{(d)}{\le}& C_{2,p}\bigg[\bigg(\frac{1}{\left|-J_1\right|}\sum\limits_{k\in -J_1}\omiga(k)\bigg)\bigg(\frac{1}{\left|-J_1\right|}\sum\limits_{k\in -J_1}\omiga(k)^{-\frac{1}{p-1}}\bigg)^{p-1}+\bigg(\frac{1}{\jsJ}\sum\limits_{k\in -J_1}\omiga(k)\bigg)\nonumber\\
&\cdot\bigg(\frac{1}{\jsJ}\sum\limits_{k\in J_2}\omiga(k)^{-\frac{1}{p-1}}\bigg)^{p-1}+\bigg(\frac{1}{\jsJ}\sum\limits_{k\in J_2}\omiga(k)\bigg)\bigg(\frac{1}{\jsJ}\sum\limits_{k\in -J_1}\omiga(k)^{-\frac{1}{p-1}}\bigg)^{p-1}+\bigg(\frac{1}{\left|J_2\right|}\sum\limits_{k\in J_2}\omiga(k)\bigg)\nonumber\\
&\cdot\bigg(\frac{1}{\left|J_2\right|}\sum\limits_{k\in J_2}\omiga(k)^{-\frac{1}{p-1}}\bigg)^{p-1}\bigg]\nonumber\\
\overset{(e)}{\le}&C_{2,p}\bigg[2\|\omiga\|_{\huaA_p(\N)}+\bigg(\frac{1}{\jsJ}\sum\limits_{k\in -J_1}\omiga(k)\bigg)\bigg(\frac{1}{\jsJ}\sum\limits_{k\in J_2}\omiga(k)^{-\frac{1}{p-1}}\bigg)^{p-1}+\bigg(\frac{1}{\jsJ}\sum\limits_{k\in J_2}\omiga(k)\bigg)\nonumber\\
&\cdot\bigg(\frac{1}{\jsJ}\sum\limits_{k\in -J_1}\omiga(k)^{-\frac{1}{p-1}}\bigg)^{p-1}\bigg],
\end{align}
\end{small}
where $(c)$ is due to $(\sum\limits_{k=1}^na_k)^p\le C_{n,p}\sum\limits_{k=1}^n\left|a_k\right|^p$, $(d)$ is from $\left|-J_1\right|\le\jsJ$ and $\left|J_2\right|\le\jsJ$, and $(e)$ is from $\omiga\in\huaA_p$ $(p>1)$.

\begin{small}
Let $\Gamma_2:=\bigg(\frac{1}{\jsJ}\sum\limits_{k\in -J_1}\omiga(k)\bigg)\bigg(\frac{1}{\jsJ}\sum\limits_{k\in J_2}\omiga(k)^{-\frac{1}{p-1}}\bigg)^{p-1}+\bigg(\frac{1}{\jsJ}\sum\limits_{k\in J_2}\omiga(k)\bigg)\bigg(\frac{1}{\jsJ}\sum\limits_{k\in -J_1}\omiga(k)^{-\frac{1}{p-1}}\bigg)^{p-1}$.
\end{small}
To estimate $\Gamma_2$, we consider the following two cases.

\textbf{Case I:} If $-J_1\subset J_2$, then
\begin{small}
\begin{equation*}\begin{array}{rcl}
\Gamma&\le&\bigg(\frac{1}{\left|J_2\right|}\sum\limits_{k\in J_2}\omiga(k)\bigg)\bigg(\frac{1}{\left|J_2\right|}\sum\limits_{k\in J_2}\omiga(k)^{-\frac{1}{p-1}}\bigg)^{p-1}+\bigg(\frac{1}{\left|J_2\right|}\sum\limits_{k\in J_2}\omiga(k)\bigg)\bigg(\frac{1}{\left|J_2\right|}\sum\limits_{k\in J_2}\omiga(k)^{-\frac{1}{p-1}}\bigg)^{p-1}\\
&\le&2\|\omiga\|_{\huaA_p(\N)}.
\end{array}\end{equation*}
\end{small}

\textbf{Case II:} If $J_2\subset -J_1\cup\{0\}=:J_3$, then
\begin{small}
\begin{equation*}\begin{array}{rcl}
\Gamma&\le&\bigg(\frac{1}{\left|J_3\right|}\sum\limits_{k\in J_3}\omiga(k)\bigg)\bigg(\frac{1}{\left|J_3\right|}\sum\limits_{k\in J_3}\omiga(k)^{-\frac{1}{p-1}}\bigg)^{p-1}+\bigg(\frac{1}{\left|J_3\right|}\sum\limits_{k\in J_3}\omiga(k)\bigg)\bigg(\frac{1}{\left|J_3\right|}\sum\limits_{k\in J_3}\omiga(k)^{-\frac{1}{p-1}}\bigg)^{p-1}\\
&\le&2\|\omiga\|_{\huaA_p(\N)}.
\end{array}\end{equation*}
\end{small}
Combining all the estimates for Case I and Case II of $\Gamma_2$, (\ref{eq2.4}) imply that
\begin{align*}
\bigg(\frac{1}{\jsJ}\sum\limits_{k\in J}\omiga(\kdjdz)\bigg)\bigg(\frac{1}{\jsJ}\sum\limits_{k\in J}\omiga(\kdjdz)^{-\frac{1}{p-1}}\bigg)^{p-1}\le4C_{2,p}\|\omiga\|_{\huaA_p(\N)}.
\end{align*}
We finish the proof of Proposition \ref{m2.15}.
\end{proof}

\begin{remark}
Let $1<p<\infty$. By \cite[Page 8]{wiener20} and Proposition \ref{m2.15}, we may obtain the relationships among weights $\huaA_p(\R_+)$, $\huaA_p(\N)$ and $\huaA_p(\Z)$:
\begin{equation*}\begin{array}{rcl}
\mu\in\huaA_p(\R_{+})\iff\{\mu(k)\}^\infty_{k=0}\in\huaA_p(\N)\Rightarrow\{\mu(\kdjdz)\}_{k\in\Z}\in\huaA_p(\Z).
\end{array}\end{equation*}
\end{remark}

Now, we recall the definition of discrete reverse H\"older class.

\begin{definition}\label{l2.1}\cite[Definition 1.4]{fu}
For $1<r<\infty$, a discrete weight $\omiga$ is said to belong to the {\it discrete reverse H\"older class} $RH_r=RH_r(\Z)$ if
\begin{equation*}
\|\omiga\|_{HR_r(\Z)}:=\sup\limits_{J\subset\Z}|J|^{1-\frac{1}{r}}\left(\sum\limits_{k\in J}\omiga(k)^r\right)^\frac{1}{r}\left(\sum\limits_{k\in J}\omiga(k)\right)^{-1}<\infty,
\end{equation*}
where $\|\omiga\|_{HR_r(\Z)}$ denotes the reverse H\"older norm of weight $\omiga$ and $J$ is any bounded interval in $\Z$.

\end{definition}

In the next proposition, we recall some basic inclusion relations of Muckenhoupt classes and reverse H\"older classes.

\begin{proposition}\label{p1.2}\cite[Lemma 2.1]{fu}
The following inclusion relations hold:
\begin{itemize}
\item[\rm(i)] If $1\le r\le p<\infty$, then $\huaA_1\subset\huaA_r\subset\huaA_p$;
\item[\rm(ii)] If $1\le r<p$, then $RH_p\subset RH_r$;
\item[\rm(iii)] If $\omiga\in\huaA_\infty$, then there exists a constant $r\in(1, \infty)$ such that $\omiga\in RH_r$.
\end{itemize}
\end{proposition}

By Proposition \ref{p1.2} $\rm(iii)$, we immediately obtain the following Proposition \ref{p1.3}, the proof of which in continuous version can be found in \cite{lu}. The same proof works here.

\begin{proposition}\label{p1.3}
If $\omiga\in\huaA_p$ $(1<p<\infty)$, then there exists a constant $\epsilon >0$ such that $p-\epsilon>1$ and $\omiga\in\huaA_{p-\epsilon}$.
\end{proposition}

By H\"older's inequality, we immediately obtain the following proposition.
\begin{proposition}\label{pro3}
Let $r>1$ and $\omiga\in RH_r$. Then for any subset $S\subset\,\text{interval}\,J\subset\Z$, it holds true that $\frac{\omiga(S)}{\omiga(J)}\le C\left(\frac{|S|}{|J|}\right)^\frac{1}{r'}$.
\end{proposition}

\section{The Calder\'on-Zygmund decomposition for weighted sequences \label{s3}}
Swarup and Alphonse \cite{bzb} discussed the discrete Calder\'on-Zygmund decomposition for $l^p$-sequences $(1\le p<\infty)$. Garcia-Cuerva and Rubio de Francia \cite{norm} discussed the continuous Calder\'on-Zygmund decomposition for weighted $L^1$-functions. Combining the methods of \cite{bzb} and \cite{norm}, we obtain the discrete Calder\'on-Zygmund decomposition for weighted $l^1$-sequences, i.e., the following Theorem \ref{t3.2}, which will be used to obtain the boundedness of discrete weighted Hardy-Littlewood maximal operators on discrete weighted Lebesgue spaces in Sect.\,\ref{s4.1}.

\begin{theorem}\label{t3.2}
Let $1\le p<\infty$, $\omiga\in\huaA_p$ and $x\in l^1_\omiga$. For every $t>0$, there exist a constant $C$ only depending on $p$ and a sequence of disjoint intervals $\mathcal{I}_t=\mathcal{I}_t(x,\omiga)$, consisting of those maximal dyadic over which the average of $|x|$ relative to $\omiga$ is greater than $t$, such that
\begin{itemize}
\item[\rm(i)] $t<\frac{1}{\omiga(I)}\sum\limits_{k\in I}\left|x(k)\right|\omiga(k)\le Ct$, $\forall$  $I\in \mathcal{I}_t$;
\item[\rm(ii)] $|x(n)|\le t$, $\forall$ $n\notin  \mathop{\cup}\limits_{I\in \mathcal{I}_t}I$;
\item[\rm(iii)] If $t_1>t_2$, then every dyadic interval in $\mathcal{I}_{t_1}$ is a subinterval of some dyadic interval in $\mathcal{I}_{t_2}$.
\end{itemize}
\end{theorem}

\begin{proof}
Let
\begin{equation*}
I_{N,j}=\{(j-1)2^N+1,\cdots,j2^N\},\quad N\in\N,\,j\in\Z.
\end{equation*}
The set of all intervals which are of the form $I_{N,j}$, $N\in\N$ and $j\in\Z$, are called
dyadic intervals. We denote by $\mathcal{I}_t=\mathcal{I}_t(x,\omiga)$ the collection formed by the maximal dyadic intervals $I$ satisfying the condition
\begin{equation}\label{eq1.2}
t<\frac{1}{\omiga(I)}\sum\limits_{k\in I}|x(k)|\omiga(k).
\end{equation}
Every dyadic interval satisfying (\ref{eq1.2}) is contained in some member of $\mathcal{I}_t$. Pick a maximal dyadic interval $I\in \mathcal{I}_t$. Then there exist some $N\in\N$ and some $j\in\Z$ such that $I=I_{N,j}$. Furthermore, there exists dyadic interval $I_{N+1,m}\supset I_{N,j}$ such that
\begin{equation*}
\frac{1}{\omiga(I_{{N+1},m})}\sum\limits_{I_{{N+1},m}}|x(k)|\omiga(k)\le t.
\end{equation*}
By simple geometric observation, we have $I_{{N+1},m}=2LI_{N,j}$ or $I_{{N+1},m}=2RI_{N,j}$, where
\begin{equation*}
\begin{aligned}
&2LI_{N,j}:=\{(j-2)2^N+1,\cdots,j2^N\},\\
&2RI_{N,j}:=\{(j-1)2^N+1,\cdots,(j+1)2^N\}.
\end{aligned}
\end{equation*}
From Proposition \ref{m2.9}, it follows that $\omiga(I_{{N+1},m})=\omiga(2LI_{N,j})\le C\omiga(I_{N,j})$ or $\omiga(I_{{N+1},m})=\omiga(2RI_{N,j})\le C\omiga(I_{N,j})$. Therefore
\begin{equation*}
\frac{1}{\omiga(I_{N,j})}\sum\limits_{k\in I_{N,j}}|x(k)|\omiga(k)\le \frac{C}{\omiga(I_{{N+1},m})}\sum\limits_{k\in I_{N+1,m}}|x(k)|\omiga(k)\le Ct.
\end{equation*}
Thus, for every  $I\in\mathcal{I}_t$, we obtain
\begin{equation*}
t<\frac{1}{\omiga(I)}\sum\limits_{k\in I}|x(k)|\omiga(k)\le Ct.
\end{equation*}
This proves $\rm(i)$.

$\rm(ii)$ Now all elements not included in $\mathop{\cup}\limits_{I\in \mathcal{I}_t}I$ form a set $S$ such that for every $n\in S$, we have $|x(n)|\le t$. Otherwise, if $|x(n)|>t$ for some $n\in S$, then we have
$$
|x(n)|=\frac{1}{\omiga(\{n\})}\sum\limits_{k\in\{n\}}|x(k)|\omiga(k)>t,
$$
thus $\{n\}\subset\mathop{\cup}\limits_{I\in \mathcal{I}_t}I$, which contradict with $n\in S$. This proves $\rm(ii)$.

$\rm(iii)$ If $t_1>t_2>0$, then for any maximal dyadic interval $I=I^{t_1}_{N,j}\in\mathcal{I}_{t_1}$, it holds true that
\begin{equation*}
t_2<t_1<\frac{1}{\omiga(I^{t_1}_{N,j})}\sum\limits_{k\in I^{t_1}_{N,j}}|x(k)|\omiga(k),
\end{equation*}
which implies that $I^{t^1}_{N,j}\in \mathcal{I}_{t_2}$ or $I^{t_1}_{N,j}$ is a subinterval of some maximal dyadic interval in $\mathcal{I}_{t_2}$.
\end{proof}

\section{The weighted estimates for discrete maximal operators \label{s4}}
\subsection{Estimates for discrete weighted Hardy-Littlewood maximal operators on discrete weighted Lebesgue spaces}\label{s4.1}

\begin{definition}\label{d2.16}
Let $x=\{x(k)\}_{k\in\Z}\subset\R$ be a sequence and $\omiga$ be a discrete weight. The {\it discrete weighted Hardy-Littlewood maximal operator} $\M_{\omiga}$ is defined by
\begin{equation*}
\M_{\omiga}x(m):=\sup_{N\in \N}\frac{1}{\omiga(\Smn)}\sum_{k\in \Smn}\left|x(k)\right|\omiga(k),\quad m\in \Z.
\end{equation*}
\end{definition}

\begin{theorem}\label{th4.1}
\begin{itemize}
\item[\rm(i)] Let $\omiga\in\huaA_1$. If $x\in l^1_\omiga$, then for every $t>0$,
\begin{equation*}
\omiga(\{k\in\Z:\M_\omiga x(k)>t\})\le C\|x\|_{l^1_\omiga},
\end{equation*}
where $C$ is a positive constant depending on $p$ and $\|\omiga\|_{\huaA_p(\Z)}$.
\item[\rm(ii)] Let $\omiga\in\huaA_p$. If $x\in l^p_\omiga$, $1< p<\infty$, then $\M_\omiga x\in l^p_\omiga$ and
\begin{equation*}
\|\M_\omiga x\|_{l^p_\omiga}\le C\|x\|_{l^p_\omiga},
\end{equation*}
where $C$ is a positive constant depending on $p$ and $\|\omiga\|_{\huaA_p(\Z)}$.
\end{itemize}
\end{theorem}

\begin{remark}
When $\Z$ is restricted to $\Z_+$, Theorem \ref{th4.1}$\rm(ii)$ coincides with \cite[Lemma 3]{lsq}. In addition, Theorem \ref{th4.1}$\rm(ii)$ is proved by means of a discrete Calder\'on-Zygmund decomposition for weighted $l^1$-sequences, while \cite[Lemma 3]{lsq} is proved via Yang inequality and Hardy operator.
\end{remark}

In order to prove Theorem \ref{th4.1}, we need a lemma for preparation.

\begin{lemma}\label{l3.11}
Let $x\in l{^p_\omiga}$ and $p>0$. Then $\|x\|^p_{l^p_\omiga}=p\int^\infty_0 \lambda^{p-1}\sum\limits_{{\{k:\left|x(k)\right|>\lambda}\}}\omiga(k)d\lambda$.
\end{lemma}

\begin{proof}
\begin{equation*}
\begin{aligned}
\|x\|^p_{l^p_\omiga}&=\sum\limits_{k\in\Z}\left|x(k)\right|^p\omiga(k)\\
&= p\sum\limits_{k\in\Z}\left(\int^{\left|x(k)\right|}_0\lambda^{p-1}d\lambda\right)\omiga(k)\\
&= p\sum\limits_{k\in\Z}\left(\int^\infty_0\lambda^{p-1}\kaf_{\{\left|x(k)\right|>\lambda\}}(\lambda)d\lambda\right)\omiga(k)\\
&= p\int^\infty_0\lambda^{p-1}\sum\limits_{k\in \Z}\kaf_{\{\left|x(k)\right|>\lambda\}}(k)\omiga(k)d\lambda\\
&= p\int^\infty_0\lambda^{p-1}\sum\limits_{{\{k:\left|x(k)\right|>\lambda}\}}\omiga(k)d\lambda.
\end{aligned}
\end{equation*}
We finish the proof of Lemma \ref{l3.11}.
\end{proof}

\begin{proof}[Proof of Theorem \ref{th4.1}]
$\rm(i)$ First, let us show
\begin{equation}\label{eq1.3}
\{k\in\Z:\M_\omiga x(k)>C_1t\}\subset\mathop{\cup}\limits_{I\in\mathcal{I}_t}3I.
\end{equation}
For any $j\notin \mathop{\cup}\limits_{I\in\mathcal{I}_t}3I$, we only need to prove $j\notin\{k\in\Z:\M_\omiga x(k)>C_1t\}$, where $C_1$ is some positive constant depending on $p$ and $\|\omiga\|_{\huaA_p(\Z)}$ to be fixed later. Let $J$ be any interval centered on $j$. Choose $N\in \N$ such that $2^N\le |J|<2^{N+1}$. Then $J$ intersects exactly 2 dyadic intervals $R^{N+1}_1$ and $R^{N+1}_2$. Assume $R^{N+1}_1$ intersects $J$ on the left and $R^{N+1}_2$ intersects $J$ on the right. Since $j\notin \mathop{\cup}\limits_{I\in\mathcal{I}_t}3I$, then $j\notin 2RI$ and $j\notin 2LI$ for any $I\in\mathcal{I}_t$. But $j\in 2RR^{N+1}_1$ and $j\in 2LR^{N+1}_2$. Therefore, both $R^{N+1}_1$ and $R^{N+1}_2$ cannot be any one of $\mathcal{I}_t$.
Hence the weighted average of $|x|$ on $R^{N+1}_1$ and $R^{N+1}_2$ are at most $t$. From $R^{N+1}_i\subset 5J, i=1,2$, and Proposition \ref{m2.9} (i), it follows that
\begin{small}
\begin{align*}
&\frac{1}{\omiga(J)}\sum\limits_{s\in J}|x(s)|\omiga(s)\\
\le &\frac{1}{\omiga(J)}\left(\sum\limits_{s\in R^{N+1}_1}|x(s)|\omiga(s)+\sum\limits_{s\in R^{N+1}_2}|x(s)|\omiga(s)\right)\\
=&\frac{\omiga(R^{N+1}_1)}{\omiga(J)}\left(\frac{1}{\omiga(R^{N+1}_1)}\sum\limits_{s\in R^{N+1}_1}|x(s)|\omiga(s)\right)
+\frac{\omiga(R^{N+1}_2)}{\omiga(J)}\left(\frac{1}{\omiga(R^{N+1}_2)}\sum\limits_{s\in R^{N+1}_2}|x(k)|\omiga(s)\right)\\
\le &\left(\frac{15}{2}\|\omiga\|_{\huaA_r(\Z)}\right)^r(t+t)\\
=: &C_1t,
\end{align*}
\end{small}
which implies (\ref{eq1.3}) holds true.

Furthermore, by (\ref{eq1.3}) and Proposition \ref{m2.9} (ii), we have
\begin{small}
\begin{equation}\label{eq1.4}
\sum\limits_{\{k\in\Z:\M_\omiga x(k)>C_1t\}}\omiga(k)\le \sum\limits_{k\in\mathop{\cup}\limits_{I\in\mathcal{I}_t}3I}\omiga(k)\le \sum\limits_{I\in\mathcal{I}_t}\omiga(3I)\le \left(3\|\omiga\|_{\huaA_r(\Z)}\right)^r\sum\limits_{I\in\mathcal{I}_t}\omiga(I)=: C_2\sum\limits_{I\in\mathcal{I}_t}\omiga(I).
\end{equation}
\end{small}

By the Calder\'on-Zygmund decomposition at height $t$ for $x$ and any $I\in\mathcal{I}_t$ (Theorem \ref{t3.2}), we have
\begin{equation*}
t<\frac{1}{\omiga(I)}\sum\limits_{k\in I}|x(k)|\omiga(k).
\end{equation*}
From this and (\ref{eq1.4}), it follows that
\begin{equation}\label{eq1.1}
\omiga(\{k\in\Z:\M_\omiga x(k)>C_1t\})\le \frac{C_2}{t}\sum\limits_{I\in\mathcal{I}_t}\sum\limits_{k\in I}|x(k)|\omiga(k)\le \frac{C_2}{t}\sum\limits_{k\in\Z}|x(k)|\omiga(k).
\end{equation}
Replacing $x$ by $C_1x$, we can obtain weak type (1,1) boundedness result of $\M_\omiga x$, i.e., $\rm(i)$ holds true.

$\rm(ii)$ Let $x=x_1+x_2$, with $x_1(k):=x(k)$ if $|x(k)|>\frac{t}{2}$, and $x_1(k):=0$ otherwise. Then $\M_\omiga x(k)\le\M_\omiga x_1(k)+\M_\omiga x_2(k)\le \M_\omiga x_1(k)+\frac{t}{2}$. By this and (\ref{eq1.1}), we obtain
\begin{align*}
\omiga(\{k\in\Z:\M_\omiga x(k)>t\})&\le\omiga\left(\left\{k\in\Z:\M_\omiga x_1(k)>\frac{t}{2}\right\}\right)\\
&\le\frac{C_2}{\frac{t}{2C_1}}\sum\limits_{k\in\Z}|x_1(k)|\omiga(k)\\
&=:\frac{C}{t}\sum\limits_{\{k:|x(k)|>\frac{t}{2}\}}|x(k)|\omiga(k).
\end{align*}
From this and Lemma \ref{l3.11}, it follows that
\begin{equation*}
\begin{aligned}
\sum\limits_{k\in\Z}|\M_\omiga x(k)|^p\omiga(k)&=p\int^\infty_0t^{p-1}\sum\limits_{\{k:\M_\omiga x(k)>t\}}\omiga(k)dt\\
&\le p\int^\infty_0t^{p-1}\frac{C}{t}\sum\limits_{\{k:|x(k)|>\frac{t}{2}\}}|x(k)|\omiga(k)dt\\
&=Cp\sum\limits_{k\in\Z}\int^{2|x(k)|}_0t^{p-2}dt|x(k)|\omiga(k)\\
&=\frac{C2^{p-1}p}{p-1}\sum\limits_{k\in\Z}|x(k)|^p\omiga(k).
\end{aligned}
\end{equation*}
This finish the proof of Theorem \ref{th4.1}.
\end{proof}

\subsection{Estimate for discrete Hardy-Littlewood maximal operators on discrete weighted Morrey spaces}\label{s4.2}
The $(l^p_q,l^\infty)$-boundedness of Hardy-Littlewood maximal operators is obtained by Gunawan and Schwanke \cite[Lemma 3.1]{jdszzlsm}. The following Theorem \ref{t3.1} shows the $(l^p_\infty,l^\infty)$-boundedness of Hardy-Littlewood maximal operators.

\begin{theorem}\label{t3.1}
Let $1<p<\infty$ and $\omiga\in\huaA_p$. If $x\in l{^p_\infty}_{(\omiga)}$, then $\M x\in l{^\infty}$ and there exists a positive constant $C$ such that $\|\M x\|_{l^\infty}\le C\|x\|_{l{^p_\infty}_{(\omiga)}}$.
\end{theorem}

\begin{proof}
Let $x\in l{^p_\infty}_{(\omiga)}$ and $m^*\in \Z$. Then by H\"older's inequality and the definition of discrete weight $\huaA_p$, we have
\begin{align*}
\M x(m^*)=&\sup\limits_{N\in\N}\frac{1}{2N+1}\sum\limits_{k\in S_{m^*,N}}\left|x(k)\right|\\
\le& \sup\limits_{m\in\Z,N\in\N}\frac{1}{2N+1}\sum\limits_{k\in S_{m,N}}\left|x(k)\right|\omiga(k)^{\frac{1}{p}}\omiga(k)^{-\frac{1}{p}}\\
\le& \sup\limits_{m\in\Z,N\in\N}\frac{1}{2N+1}\bigg(\sum\limits_{k\in S_{m,N}}\left|x(k)\right|^p\omiga(k)\bigg)^\pfzy\bigg(\sum\limits_{k\in S_{m,N}}\omiga(k)^{-\frac{p'}{p}}\bigg)^{\frac{1}{p'}}\\
\le& C\sup\limits_{m\in\Z,N\in\N}\bigg(\sum\limits_{k\in S_{m,N}}\left|x(k)\right|^p\omiga(k)\bigg)^\pfzy\bigg(\sum\limits_{k\in S_{m,N}}\omiga(k)\bigg)^{-\frac{1}{p}}\\
=& C\|x\|_{l{^p_\infty}_{(\omiga)}}.
\end{align*}
We finish the proof of Theorem \ref{t3.1}.
\end{proof}

\begin{theorem}\label{t3.3}
Let $1<p<q<\infty$ and $\omiga$ be a discrete weight on $\Z$.
\begin{itemize}
\item[\rm(i)]  If $\omiga\in\huaA_p$, then $\M$ is bounded  from ${l^p_q}_{(\omega)}$ to ${l^p_q}_{(\omega)}$.
\item[\rm(ii)] If $\M$ is bounded from ${l^p_q}_{(\omega)}$ to ${l^p_q}_{(\omega)}$, then $\omega\in\huaA_q$.
\end{itemize}
\end{theorem}

\begin{remark}
When $\omiga\equiv 1$, Theorem \ref{t3.3}(i) coincides with \cite[Theorem 3.2]{jdszzlsm}.
\end{remark}

In order to prove Theorem \ref{t3.3}, we need following lemmas.

\begin{lemma}\label{t3.12}
Let $1<p\le q<\infty$, $\omiga\in\huaA_\infty$ and $x\in {l^p_q}_{(\omiga)}$. Then there exists a positive constant $C$ such that
\begin{equation*}
\|\M_\omiga x\|_{{l^p_q}_{(\omiga)}}\le C\|x\|_{{l^p_q}_{(\omiga)}}.
\end{equation*}
\end{lemma}

\begin{proof}
The proof of Lemma \ref{t3.12} is similar to that of the continuous cases (see the proof of \cite[Theorem 3.1]{jdM}) by using discrete H\"older's inequality, Theorem \ref{th4.1}$\rm(ii)$ and Proposition \ref{m2.9}$\rm(i)$. The details being omitted.
\end{proof}

\begin{lemma}\label{4.1}
Let $1<p<q<\infty$ and $\omiga$ be a discrete weight. If $\M$ is a bounded from $l^q_\omiga$ to ${wl^p_q}_{(\omiga)}$, then $\omiga\in\huaA_q$.
\end{lemma}

\begin{proof}
The proof of Lemma \ref{4.1} is similar to that of the continuous cases (see the proof of \cite[Theorem 1.1]{wang})
\end{proof}

\begin{proof}[Proof of Theorem \ref{t3.3}]
$\rm(i)$ If $\omiga\in\huaA_p$, then by Proposition \ref{p1.3}, there exists $\epsilon>0$ such that $1<p-\epsilon<p$ and $\omiga\in\huaA_{p-\epsilon}$. Let $r:=p-\epsilon$. By H\"older's inequality and the definition of $\huaA_r$ weights, we obtain
\begin{equation*}\begin{array}{rcl}
\frac{1}{\jssmn}\sum\limits_{k\in \Smn}\left|x(k)\right|&=& \frac{1}{\jssmn}\sum\limits_{k\in \Smn}\left|x(k)\right|\omiga(k)^{\frac{1}{r}}\omiga(k)^{-\frac{1}{r}}\\
&\le& \frac{1}{\jssmn}\left(\sum\limits_{k\in \Smn}\left|x(k)\right|^r\omiga(k)\right)^\frac{1}{r}\left(\sum\limits_{k\in \Smn}\omiga(k)^{-\frac{r'}{r}}\right)^\frac{1}{r'}\\
&\le& C\left(\frac{1}{\omiga(\Smn)}\sum\limits_{k\in \Smn}\left|x(k)\right|^r\omiga(k)\right)^{\frac{1}{r}}.
\end{array}\end{equation*}
Taking the supremum in the last inequality for all $N\in\N$, we obtain
\begin{equation*}
\begin{aligned}
\M x(m)\le & C\sup\limits_{N\in\N}\left(\frac{1}{\omiga(\Smn)}\sum\limits_{k\in\Smn}\left|x(k)\right|^r\omiga(k)
\right)^\frac{1}{r}\\
= & C(\M_\omiga x^r(m))^\frac{1}{r}.
\end{aligned}
\end{equation*}
From this and Lemma \ref{t3.12}, it follows that
\begin{align*}
\omiga(\Smn)^{\frac{1}{q}-\frac{1}{p}}\left(\sum\limits_{k\in S_{m,N}}\left|\M x(k)\right|^p\omiga(k)\right)^\frac{1}{p}
\le& C\omiga(\Smn)^{\frac{1}{q}-\frac{1}{p}}\left(\sum\limits_{k\in S_{m,N}}\left|\M_\omiga x^r(k)\right|^\frac{p}{r}\omiga(k)\right)^\frac{1}{p}\\
\le& C\|\M_\omiga x^r\|{^{\frac{1}{r}}_{l{^\frac{p}{r}_{\frac{q}{r}}}_{(\omiga)}}}
\le C\|x^r\|{^{\frac{1}{r}}_{l{^\frac{p}{r}_{\frac{q}{r}}}_{(\omiga)}}}
=C\|x\|_{l{^p_q}_{(\omiga)}}.
\end{align*}

$\rm(ii)$ If $\M$ is bounded from ${l^p_q}_{(\omiga)}$ to ${l^p_q}_{(\omiga)}$, by Proposition \ref{p2.2}, it follows that $\M$ is also bounded from ${l^p_q}_{(\omiga)}$ to ${wl^p_q}_{(\omiga)}$, which together with Proposition \ref{m2.4}(iii), it follows that $\M$ is also bounded from $l^q_\omiga$ to ${wl^p_q}_{(\omiga)}$. By this and Lemma \ref{4.1}, we obtain $\omiga\in\huaA_q$.
We finish the proof of Theorem \ref{t3.3}.
\end{proof}

\begin{remark}\label{r4.2}
Let $x=\{x(k)\}_{k\in\Z}$ be a sequence, $\omiga'$ be a discrete weight on $\Z$ and let
\begin{equation}\label{eq4.3}
\omiga(k):=
\begin{cases}
|k|^\beta, \quad k\neq 0,\\
~1~~\,, \quad k=0.
\end{cases}
\end{equation}
be a discrete power weight on $\Z$. Then for $1<p<\infty$, $\omiga\in\huaA_p$ if and only if $-1<\beta<{p-1}$ (Continuous case see \cite[Page 141]{shu}). Therefore, when (\ref{eq4.3}) with $-1<\beta<{p-1}$, Theorem \ref{t3.3}(i) still holds true. In particular, when $\omiga'\in\huaA_q$ by Proposition \ref{p1.3}, there exists $\epsilon>0$ such that $\omiga'\in\huaA_{q-\epsilon}$. If $p=q-\epsilon$, then by Theorem \ref{t3.3} we obtain that $\M$ is bounded from ${l^p_q}_{(\omiga')}$ to ${l^p_q}_{(\omiga')}$ if and only if $\omiga'\in\huaA_p$.
\end{remark}

Let $1<p<q$. If $\omiga\in\huaA_p$, then $\M$ is bounded on ${l^p_q}_{(\omiga)}$. Conversely, if $\M$ is bounded on ${l^p_q}_{(\omiga)}$, then $\omiga\in\huaA_q$. However, in the following subsection, when $\omiga$ is discrete power weight on $\Z$ and $\beta$ satisfies certain conditions, we can obtain the sufficient and necessary conditions for the boundedness of $\M$ on ${l^p_q}_{(\omiga)}$.
\subsubsection{The power weight cases}\label{s4.2.2}
\begin{theorem}\label{t4.1}
Let $1<p<q<\infty$ and let
\begin{equation*}
\omiga(k):=
\begin{cases}
|k|^\beta, \quad k\neq 0,\\
~1~~\,, \quad k=0.
\end{cases}
\end{equation*}
The discrete Hardy-Littlewood maximal operator is bounded on ${l^p_q}_{(\omiga)}$ if and only if $-1<\beta<q-1$.
\end{theorem}

To prove Theorem \ref{t4.1}, we need some properties and lemmas. For $1\le p<\infty$ and nonnegative $u$ and $\omiga$, define
\begin{equation*}
\|x\|_{{l^p}_{(u,\omiga)}}:=\sup_{m\in\Z,N\in\Z}\left(\frac{1}{u(S_{m,N})}\sum\limits_{k\in S_{m,N}}|x(k)|^p\omiga(k)\right)^\frac{1}{p}.
\end{equation*}
Lighted by \cite{2020D}, we classify the symmetric interval into three types:
\begin{itemize}
\item[\rm{I}:] symmetric intervals centered at the origin;
\item[\rm{II}:] symmetric intervals $S_{m,N}$ centered at $m\neq 0$ and $N\le\lfloor {\frac{|m|}{2}} \rfloor$, where $\lfloor t \rfloor$ is the nearest integer less than or equal to $t$;
\item[\rm{III}:] symmetric intervals $S_{m,N}$ centered at $m\neq 0$ and $N>\lfloor {\frac{|m|}{2}} \rfloor$.
\end{itemize}

\begin{proposition}\label{pro1}
Let $u,\omiga$ be two discrete weights and let $x=\{x(k)\}_{k\in\Z}$ be a sequence.
\begin{itemize}
\item[\rm(i)] If $u$ is doubling, i.e., there exists a constant $C>0$ such that for any symmetric interval $S_{m,N}$, we have $u(2S_{m,N})\le Cu(S_{m,N})$, we get a value equivalent to $\|x\|_{{l^p}_{(u,\omiga)}}$ by taking the supremum only on symmetric intervals of types $\rm{I}$ and $\rm{II}$.
\item[\rm(ii)] If moreover $u$ satisfies $\sum\limits^\infty_{j=0}u(S_{0,a^jN})\le Cu(S_{0,N})$ for some $a<1$ and $C$ independent of $N$, then we can further restrict the supremum to symmetric intervals of type II.
\item[\rm(iii)] If $1<p<q<\infty$, $\omiga\in\huaA_q$ and $u:=\omiga^{1-\frac{p}{q}}$, then we get a value equivalent to $\|x\|_{{l^p}_{(u,\omiga)}}$ by taking the supremum only on symmetric intervals of type $\rm{II}$.
\end{itemize}
\end{proposition}

\begin{proof}
Since the proofs of $\rm(i)$ and $\rm(ii)$ are almost the same as the one of \cite[Proposition 2.1]{2020D}, we omit it. We only prove $\rm(iii)$. In fact, if $\omiga\in\huaA_q$, then by Proposition \ref{p1.2}, there exists a constant $r>1$ such that $\omiga\in RH_r$. From this, Proposition \ref{pro3}, $1<p<q$ and $0<a<1$, it follows that $\sum\limits^\infty_{j=0}u(S_{0,a^jN})\le\sum\limits^\infty_{j=0}a^{j\frac{1}{r'}(1-\frac{p}{q})}u(S_{0,N})\le Cu(S_{0,N})$. Then the assumption of $\rm(ii)$ holds and hence we finish the proof of Proposition \ref{pro1}.

\end{proof}


\begin{proposition}\label{pro2}
Let $1<p<q<\infty$, $\beta>-1$ and let
\begin{equation*}
\omiga(k):=
\begin{cases}
|k|^\beta, \quad k\neq 0,\\
~1~~\,, \quad k=0.
\end{cases}
\end{equation*}
If $\M$ is bounded from ${l^p_q}_{(\omiga)}$ to ${wl^p_q}_{(\omiga)}$, then $-1<\beta<q-1$.
\end{proposition}

\begin{proof}
Let $x:=\kaf_{S_{0,N}}$ and $S:=S_{m,l}$ with $|m|\ge 1$ and $l\le\lfloor {\frac{|m|}{2}} \rfloor$. We can assume $|m|\le 2N$, otherwise $S_{m,l}\cap S_{0,N}=\varnothing$. For $k\in S$, we have
\begin{equation*}
\M x(k)\ge\frac{1}{|S_{0,3N}|}\sum\limits_{s\in S_{0,3N}}\kaf_{S_{0,N}}(s)=\frac{|S_{0,N}|}{|S_{0,3N}|}\ge\frac{2}{7}.
\end{equation*}
If $t<\frac{2}{7}$, then we have $S\subset\{k\in S: \M x(k)>t\}$. Assuming that $\M$ is bounded from ${l^p_q}_{(\omiga)}$ to ${wl^p_q}_{(\omiga)}$, then from this and Proposition \ref{m2.4}(iii), it follows that
\begin{equation*}
\begin{aligned}
\left(\frac{t^p\omiga(S)}{\omiga(S)^{1-\frac{p}{q}}}\right)^\frac{1}{p}\le
\left(\frac{t^p\omiga(\{k\in S:\M x(k)>t\})}{\omiga(S)^{1-\frac{p}{q}}}\right)^\frac{1}{p}\le C\|\kaf_{S_{0,N}}\|_{{l^p_q}_{(\omiga)}}
\le C\|\kaf_{S_{0,N}}\|_{l^q_\omiga}.
\end{aligned}
\end{equation*}
Let $t$ tend to $\frac{2}{7}$ and since $|k|\sim |m|$ for $k\in S_{m,l}$, we need
\begin{equation}\label{eq4.4}
\begin{aligned}
&\sup\limits_{1\le|m|\le 2N}\sup\limits_{0\le l\le\lfloor {\frac{|m|}{2}} \rfloor}N^{-\frac{\beta+1}{q}}|m|^\frac{\beta}{q}(2l+1)^\frac{1}{q}< \infty\\
\Longleftarrow~~~&\sup\limits_{1\le|m|\le 2N}N^{-\frac{\beta+1}{q}}|m|^\frac{\beta}{q}(2|m|)^\frac{1}{q}< \infty\\
\Longleftarrow~~~& 2^\frac{1}{q}\sup\limits_{1\le|m|\le 2N}N^{-\frac{\beta+1}{q}}|m|^{\frac{1}{q}(\beta+1)}<\infty\\
\stackrel{\text{if}~\beta+1>0}\Longleftarrow~& 2^\frac{1}{q}N^{-\frac{\beta+1}{q}}(2N)^{\frac{\beta+1}{q}}\le C.
\end{aligned}
\end{equation}
Therefore, if $\beta> -1$, (\ref{eq4.4}) holds true.

On the other hand, by the proof of Theorem \ref{t3.3}(ii), we obtain that if $\M$ is bounded from ${l^p_q}_{(u)}$ to ${wl^p_q}_{(u)}$, then $u\in\huaA_q$. Applying this to power weight $\omiga(k)$, we have
\begin{equation*}
\beta<q-1.
\end{equation*}
Combine all of the above estimates, we obtain
\begin{equation*}
-1<\beta<q-1,
\end{equation*}
which completes the proof of Proposition \ref{pro2}.
\end{proof}

The proof of following lemma in continuous version can be found in \cite[Lemma 2.12]{2020D}. The same proof method also works here.

\begin{lemma}\label{l4.1}
Let $1<p<q<\infty$, $x=\{x(k)\}_{k\in\Z}$ be a nonnegative sequence and let
\begin{equation*}
\omiga(k):=
\begin{cases}
|k|^\beta, \quad k\neq 0,\\
~1~~\,, \quad k=0.
\end{cases}
\end{equation*}
\begin{itemize}
\item[\rm(i)] Let $S=S_{m,N}$ be symmetric interval such that $|m|\ge 1$ and $N\le\lfloor {\frac{|m|}{2}} \rfloor$. Then for arbitrary $\beta$ it holds
    \begin{equation}\label{eq4.1}
    \frac{1}{|S_{m,N}|}\sum\limits_{k\in S_{m,N}}x(k)\le C|m|^{-\frac{\beta}{q}}|S_{m,N}|^{-\frac{1}{q}}\|x\|_{{l^p_q}_{(\omiga)}}.
    \end{equation}
\item[\rm(ii)] Let $S=S_{0,N}$ and $\beta>-1$. Then for $\beta<q-1$ it holds
\begin{equation}\label{eq4.2}
\frac{1}{|S_{0,N}|}\sum\limits_{k\in S_{0,N}}x(k)\le C|S_{0,N}|^{-\frac{1+\beta}{q}}\|x\|_{{l^p_q}_{(\omiga)}}.
\end{equation}
\end{itemize}
\end{lemma}

Now we shall prove Theorem \ref{t4.1}.

\begin{proof}[Proof of Theorem \ref{t4.1}]
When $-1<\beta<q-1$, by Remark \ref{r4.2}, we have $\omiga\in\huaA_q$. Then we can limit ourselves to consider symmetric intervals of the form $S:= S_{m,N}$ with $|m|\ge 1$ and $N\le\lfloor {\frac{|m|}{2}} \rfloor$ (see Proposition \ref{pro1}(iii)). The necessity holds by Proposition \ref{pro2}. For the sufficiency we let $x=\{x(k)\}_{k\in\Z}$ be a nonnegative sequence, and let $x=x_1+x_2$, where $x_1:=x\kaf_{2S_{m,N}}$. Then $\M x(k)\le\M x_1(k)+\M x_2(k)$.

To deal with $\M x_1$, by $|k|\sim |m|$ for $k\in S_{m,N}$ and Theorem \ref{1.1}, we obtain
\begin{equation*}
\sum\limits_{k\in S_{m,N}}\M x_1(k)^p|k|^\beta\le |m|^\beta\sum\limits_{k\in\Z}\M x_1(k)^p\le C|m|^\beta\sum\limits_{k\in 2S_{m,N}}x(k)^p\le C\sum\limits_{k\in 2S_{m,N}}x(k)^p|k|^\beta.
\end{equation*}

On the other hand, since there exists a symmetric interval $S'=S_{m',N'}\supset S=S_{m,N}$ such that
\begin{equation}\label{eq4.6}
\frac{1}{2}\M x_2(k)<\frac{1}{|S'|}\sum\limits_{j\in S'}x_2(j)\quad\text{for all}\,~k\in S.
\end{equation}

If $S'$ is a symmetric interval of type $\rm{II}$, i.e., $|m'|\ge 1$ and $N'\le\lfloor {\frac{|m'|}{2}} \rfloor$, then we have $|m'|\sim |m|$. From this, (\ref{eq4.6}) and (\ref{eq4.1}), it follows that
\begin{equation}\label{eq4.7}
\begin{aligned}
&\omiga(S_{m,N})^{\frac{p}{q}-1}\sum\limits_{k\in S_{m,N}}\M x_2(k)^p\omiga(k)\sim \left(|m|^\beta(2N+1)\right)^{\frac{p}{q}-1}\sum\limits_{k\in S_{m,N}}\M x_2(k)^p|k|^\beta\\
\le &2^p|m|^{\frac{p\beta}{q}-\beta}(2N+1)^{\frac{p}{q}-1}|m|^\beta(2N+1)\left(\frac{1}{|S_{m',N'}|}\sum\limits_{j\in S_{m',N'}}x_2(j)\right)^p\\
\le &C|m|^\frac{p\beta}{q}(2N+1)^\frac{p}{q}|m'|^{-\frac{p\beta}{q}}|2N'+1|^{-\frac{p}{q}}\|x\|^p_{{l^p_q}_{(\omiga)}}\\
\le &C|m|^\frac{p\beta}{q}(2N+1)^\frac{p}{q}|m|^{-\frac{p\beta}{q}}|2N+1|^{-\frac{p}{q}}\|x\|^p_{{l^p_q}_{(\omiga)}}\\
\le &C\|x\|^p_{{l^p_q}_{(\omiga)}}.
\end{aligned}
\end{equation}
Taking the supremum in (\ref{eq4.7}) for all symmetric interval of type $\rm{II}$, then we have
\begin{equation*}
\|\M x_2\|_{{l^p_q}_{(\omiga)}}\le C\|x\|_{{l^p_q}_{(\omiga)}}.
\end{equation*}

If $S'$ is a symmetric interval of type $\rm{I}$, i.e., $S'=S_{0,N'}$, then by (\ref{eq4.6}), (\ref{eq4.2}) with $-1<\beta<q-1$ and $|m|\le N'$, we have
\begin{equation}\label{eq4.8}
\begin{aligned}
&\omiga(S_{m,N})^{\frac{p}{q}-1}\sum\limits_{k\in S_{m,N}}\M x_2(k)^p\omiga(k)\sim\frac{1}{(|m|^\beta(2N+1))^{1-\frac{p}{q}}}\sum\limits_{k\in S_{m,N}}\M x_2(k)^p|k|^\beta\\
\le & C\left(|m|^\beta(2N+1)\right)^{\frac{p}{q}-1}|m|^\beta(2N+1)(2N'+1)^{-\frac{p}{q}(1+\beta)}\|x\|^p_{{l^p_q}_{(\omiga)}}\\
= & C\left(\frac{|m|}{2N'+1}\right)^{\frac{p}{q}\beta}\left(\frac{2N+1}{2N'+1}\right)^{\frac{p}{q}}\|x\|^p_{{l^p_q}_{(\omiga)}}\\
\le & C\left(\frac{|m|}{2N'+1}\right)^{\frac{p}{q}\beta}\left(\frac{3|m|}{2N'+1}\right)^{\frac{p}{q}}\|x\|^p_{{l^p_q}_{(\omiga)}}\\
\le & C\|x\|^p_{{l^p_q}_{(\omiga)}}.
\end{aligned}
\end{equation}
Taking the supremum in (\ref{eq4.8}) for all symmetric interval of type $\rm{II}$, then we have
\begin{equation}\label{eq4.9}
\|\M x_2\|_{{l^p_q}_{(\omiga)}}\le C\|x\|_{{l^p_q}_{(\omiga)}}.
\end{equation}

If $S'$ is a symmetric interval of type $\rm{III}$, i.e., $|m'|\ge 1$ and $N'>\lfloor {\frac{|m'|}{2}} \rfloor$, then we have $S_{m',N'}\subset S_{0,|m'|+N'}\subset S_{0,3N'}$ and
\begin{small}
\begin{equation*}
\frac{1}{2}\M x_2(k)<\frac{1}{|S_{m',N'}|}\sum\limits_{j\in S_{m',N'}}x(j)\le\frac{|S_{0,3N'}|}{|S_{m',N'}|}\frac{1}{|S_{0,3N'}|}\sum\limits_{j\in S_{0,3N'}}x(j)\le\frac{C}{|S_{0,3N'}|}\sum\limits_{j\in S_{0,3N'}}x(j), k\in S.
\end{equation*}
\end{small}\\
Therefore, repeating the estimate of (\ref{eq4.9}) with the fact that $S_{0,3N'}$ is a symmetric interval of type $\rm{I}$, we may obtain $\|\M x_2\|_{{l^p_q}_{(\omiga)}}\le C\|x\|_{{l^p_q}_{(\omiga)}}$ and hence finish the proof of Theorem \ref{t4.1}.
\end{proof}

\begin{remark}
For another different weighted Morrey spaces as follows: Let $\omiga\in L_{loc}(\R^n)$ be a weight, $p\in (0,\infty)$ and $\lambda\in [0,n)$ be two real parameters. The weighted Morrey space $L^{p,\lambda}_{\omiga}(\R^n)$ is the set of $f\in L^p_{loc}(\R^n)$ satisfying
\begin{equation*}
\|f\|_{L^{p,\lambda}_\omiga}=\sup\limits_{Q\in{\mathcal Q}}\left(\frac{1}{l(Q)^\lambda}\int_Q|f(x)\omiga(x)|^pdx\right)^\frac{1}{p}<\infty.
\end{equation*}
Tanaka \cite{2015t} and Nakamura-Sawano-Tanaka \cite{2018fsc} discussed the necessary and sufficient conditions for the boundedness of the Hardy-Littlewood maximal operator, the fractional Hardy-Littlewood maximal operator and the fractional integral operator on weighted Morrey spaces $L^{p,\lambda}_\omiga$. We will consider whether these conclusions still hold true under discrete setting in our next article.
\end{remark}

\subsection*{Acknowledgment}

The authors would like to express their deep gratitude to the referees for their meticulous work and useful comments which do improve Theorems \ref{t3.3} and \ref{t4.1}, and the presentation of this article.

\bigskip\medskip

\noindent Xuebing Hao, Shuai Yang and Baode Li (Corresponding author),
\medskip

\noindent College of Mathematics and System Sciences\\
 Xinjiang University\\
 Urumqi, 830017\\
P. R. China
\smallskip

\noindent{E-mail }:\\
\texttt{1659230998@qq.com} (Xuebing Hao)\\
\texttt{2283721784@qq.com} (Shuai Yang)\\
\texttt{baodeli@xju.edu.cn} (Baode Li)\\
\bigskip \medskip

\end{document}